\documentclass[11pt]{amsart}
\usepackage{latexsym}
\usepackage{amsmath}
\usepackage{amsfonts}
\usepackage{amssymb}
\usepackage{enumitem}
\usepackage{blkarray}
\usepackage {tikz}
\usepackage{color}
\usepackage{comment}
\usepackage{mathtools}
\usetikzlibrary {positioning}
\usepackage{blkarray, bigstrut}
\usetikzlibrary{arrows,shapes}
\usepackage{color}
\usepackage{placeins}
\usepackage[noend]{algorithmic} 
\usepackage{algorithm,caption}
\algsetup{indent=2em}

\newtheorem{theorem}{Theorem}[section]
\newtheorem{lemma}[theorem]{Lemma}
\newtheorem*{theorem*}{Theorem}
\newtheorem{corollary}[theorem]{Corollary}
\newtheorem{proposition}[theorem]{Proposition}

\newtheorem{sublemma}{}[theorem]
\newtheorem{conjecture}[theorem]{Conjecture}

\theoremstyle{definition}

\theoremstyle{remark}

\numberwithin{equation}{section}
\usepackage{stackengine}

\stackMath

\newcommand\Ga{\mathcal{G}^\mathrm{add}}

\newcommand\Ge{\mathcal{G}^\mathrm{epex}}

\begin{document}

\title[Edge-add graphs]{Structural Bounds and Forbidden Induced Subgraphs for Edge-Add Graph Classes}

\author[Singh]{Jagdeep Singh}
\address{Department of Mathematics and Statistics\\
Mississippi State University \\
Mississippi State, Mississippi}
\email{singhjagdeep070@gmail.com}

\author[Sivaraman]{Vaidy Sivaraman}
\address{Department of Mathematics and Statistics\\
Mississippi State University \\
Mississippi State, 
Mississippi}
\email{vaidysivaraman@gmail.com}

\date{\today}

\begin{abstract}
A class $\mathcal{G}$ of graphs is hereditary if it is closed under taking induced subgraphs. We investigate the edge-add class, $\Ga$, consisting of graphs that can be made members of $\mathcal{G}$ by adding at most one edge. While it is known that the operations of vertex deletion and edge deletion preserve the finiteness of forbidden induced subgraphs for classes with finite exclusions, the behavior of edge addition on classes with infinite exclusions remains largely unexplored.

We characterize the edge-add class of chordal graphs by their forbidden induced subgraphs and extend the result to a general finiteness theorem: for any fixed $p\ge0$, the set of forbidden induced subgraphs for $p$-edge-add chordal graphs that are not cycles is finite. In contrast, we show that this phenomenon does not extend to perfect graphs. Furthermore, we provide explicit structural bounds proving that edge addition preserves finiteness for base classes with finitely many exclusions. We conclude by providing the complete structural characterizations and explicit minimal obstruction lists for the edge-add classes of split and threshold graphs, and generalize these results to $(p,q)$-edge split graphs.
\end{abstract}

\maketitle

\section{Introduction}

We consider simple, finite, and undirected graphs. The notation and terminology follow \cite{text2}. An {\bf induced subgraph} of a graph $G$ is a graph $H$ that can be obtained from $G$ by a sequence of vertex deletions. Specifically, $G-v$ is the graph obtained from $G$ by deleting a vertex $v$ of $G$. For a subset $T$ of $V(G)$, we denote by $G[T]$ the induced subgraph on the vertex set $T$. We say that $G$ {\it contains} $H$ if $H$ is an induced subgraph of $G$. 

A class $\mathcal{G}$ of graphs is called {\bf hereditary} if it is closed under taking induced subgraphs. For a hereditary class $\mathcal{G}$, we call a graph $H$ a {\bf forbidden induced subgraph} for $\mathcal{G}$ if $H$ is not in $\mathcal{G}$ but every proper induced subgraph of $H$ is in $\mathcal{G}$. 

It is a well-known consequence of the Robertson-Seymour theorem that every graph class closed under minors contains only finitely many forbidden minors \cite{text2}. However, the situation with induced subgraphs is vastly different: certain hereditary graph classes afford a finite set of forbidden induced subgraphs, while others require an infinite list. A natural question is understanding which operations on graph classes preserve the finiteness of their forbidden induced subgraphs, and conversely, how these operations behave on classes with infinite exclusions. 

For a graph $G$, the graph $\overline{G}$ denotes the complement of $G$. For an edge $e$ of $G$, $G-e$ is the graph obtained by deleting $e$. For a non-edge $e$, $G+e$ is the graph obtained by adding $e$. We study the {\bf edge-add class, $\Ga$}, defined as the class of graphs $G$ such that $G \in \mathcal{G}$, or $G$ has a non-edge $e$ such that $G+e \in \mathcal{G}$. More generally, for any integer $p \ge 1$, the {\bf $p$-edge-add class} consists of graphs at most $p$ edge additions away from $\mathcal{G}$. Similarly, the {\bf edge-apex class, $\Ge$}, consists of graphs at most one edge deletion away from $\mathcal{G}$. 

It is known that modifying a hereditary class by allowing a single vertex deletion (the \textit{vertex-apex class}, denoted $\mathcal{G}(1)$) preserves the finiteness of forbidden induced subgraphs, as shown by Borowiecki, Drgas-Burchardt, and Sidorowicz \cite{boro}. While it is easy to observe that both $\Ge \subseteq \mathcal{G}(1)$ and $\Ga \subseteq \mathcal{G}(1)$, the finiteness of obstructions for a superclass does not imply the finiteness of obstructions for a subclass. The authors in \cite{vsin} showed that if $\mathcal{G}$ has finitely many forbidden induced subgraphs, then so does $\Ge$. 
In Section 2, we establish the corresponding result and strict structural bounds for $\Ga$ via complementation.

\begin{theorem}
\label{edge_add_bound_intro}
Let $\mathcal{G}$ be a hereditary class of graphs, and let $c$ and $t$ denote the maximum number of vertices and non-edges, respectively, over all forbidden induced subgraphs for $\mathcal{G}$. If $G$ is a forbidden induced subgraph for $\Ga$, then $|V(G)| \leq \max{\{2c, c + t(c-2)\}}$. 
\end{theorem}

While the behavior of these local modifications is thus resolved for classes with finite obstructions, their effect on classes with infinite obstructions remains largely unexplored. A graph is \textbf{chordal} if it contains no induced cycle of length four or greater. The class of chordal graphs has infinitely many forbidden induced subgraphs. One of our main results, Theorem \ref{edge_add_chordal_intro}, is that the set of forbidden induced subgraphs for the edge-add class of chordal graphs that are not cycles is finite. We prove the following in Section 3.

\begin{theorem}
\label{edge_add_chordal_intro}
Let $G$ be a forbidden induced subgraph for the class of edge-add chordal graphs. Then $G$ is a cycle of length at least five, or $6 \leq |V(G)| \leq 8$, and $G$ is isomorphic to one of the finite set of graphs in Figures \ref{chordal_6_7} and \ref{chordal_8}.
\end{theorem}

Remarkably, this extends to $p$-edge-add chordal graphs for any fixed $p \geq 0$, proving that modifying chordal graphs by any finite number of edge additions yields a strictly finite set of non-cycle obstructions.

\begin{theorem}
\label{p_edge_add_chordal_intro}
Let $\mathcal{F}$ be the set of all forbidden induced subgraphs for the class of $p$-edge-add chordal graphs that are not cycles. Then $\mathcal{F}$ is finite. 
\end{theorem}

The \textit{chromatic number} of a graph $G$ is denoted by $\chi(G)$, and its \textit{clique number} (size of its largest clique) is denoted by $\omega(G)$. A graph $G$ is \textbf{perfect} if $\chi(H) = \omega(H)$ for all induced subgraphs $H$ of $G$. In contrast, we demonstrate that this finiteness phenomenon does not extend to the edge-add class of perfect graphs. 

Finally, in Section 4, we provide full structural characterizations and explicit minimal obstructions for the edge-add and edge-apex classes of two classic families defined by finitely many exclusions: split graphs and threshold graphs. While a direct application of Theorem \ref{edge_add_bound_intro} to split graphs establishes an upper bound of $20$ for the number of vertices in a forbidden induced subgraph, we analytically establish the following.

\begin{theorem}
\label{split_intro}
Let $G$ be a forbidden induced subgraph for the class of edge-add split graphs. Then $5 \leq |V(G)| \leq 8$, and $G$ is isomorphic to one of the graphs in Figures \ref{split_5}, \ref{split_6}, and \ref{split_7_8}.
\end{theorem}

We also prove the following for threshold graphs.

\begin{theorem}
\label{edge_add_threshold_intro}
Let $G$ be a forbidden induced subgraph for the class of edge-add threshold graphs. Then $4 \leq |V(G)| \leq 8$, and $G$ is isomorphic to one of the graphs in Figure \ref{threshold}. 
\end{theorem}

Because the operations of edge-addition and edge-deletion commute, we also extend these characterizations to $(p,q)$-edge split graphs (graphs at most $p$ edge-additions and $q$ edge-deletions away from a split graph), establishing that this generalized class also maintains a finite obstruction set.

\section{The Edge-Add Class and Structural Bounds}

For a class $\mathcal{G}$ of graphs, let $\overline{\mathcal{G}} := \{\overline{G} : G \in \mathcal{G}\}$. It is immediate that if $\mathcal{G}$ is hereditary, then so is $\overline{\mathcal{G}}$. The proofs of following two propositions follow directly from definitions. 
 
 \begin{proposition}
 \label{prop1}
 $F \not \in \mathcal{G}$ if and only if $\overline{F} \not \in \overline{\mathcal{G}}$.
 \end{proposition}

\begin{proposition}
\label{prop2}
Let $\mathcal{G}$ be a hereditary class of graphs. Then $H$ is a forbidden induced subgraph for $\Ga$ if and only if $\overline{H}$ is a forbidden induced subgraph for $\overline{\mathcal{G}}^{\mathrm{epex}}$.
\end{proposition}

An immediate consequence of Proposition \ref{prop2} is the following.

\begin{corollary}
\label{complementation}
Let $\mathcal{G}$ be a hereditary class of graphs that is closed under complementation. Then $H$ is a forbidden induced subgraph for $\Ga$ if and only if $\overline{H}$ is a forbidden induced subgraph for $\Ge$.
\end{corollary}

The following from \cite{vsin} provides the bound on the number of vertices of a forbidden induced subgraph for $\Ge$. 

\begin{theorem}
\label{previous_bound}
Let $\mathcal{G}$ be a hereditary class of graphs and let $c$ and $k$ denote the maximum number of vertices and edges in a forbidden induced subgraph for $\mathcal{G}$. If $G$ is a forbidden induced subgraph for $\Ge$, then $|V(G)| \leq \max{\{2c, c + k(c-2)\}}$. 
\end{theorem}

The following is immediate from Proposition \ref{prop2} and Theorem \ref{previous_bound}.

\begin{theorem}
\label{edge_add_bound}
Let $\mathcal{G}$ be a hereditary class of graphs and let $c$ and $t$ denote the maximum number of vertices and non-edges in a forbidden induced subgraph for $\mathcal{G}$. If $G$ is a forbidden induced subgraph for $\Ga$, then $|V(G)| \leq \max{\{2c, c + t(c-2)\}}$. 
\end{theorem}

For a fixed $q \geq 0$ and a hereditary class $\mathcal{G}$ of graphs, we say a graph $G$ is in the {\bf $q$-edge-apex class} of $\mathcal{G}$ if a graph in $\mathcal{G}$ can be obtained from $G$ by deleting at most $q$ edges of $G$. Similarly, for a fixed $p \geq 0$, a graph is in the {\bf $p$-edge-add class} of  $\mathcal{G}$ if a graph in $\mathcal{G}$ can be obtained from $G$ by adding at most $p$ non-edges of $G$.

By the repeated application of Theorem \ref{previous_bound}, we obtain the following.

\begin{corollary}
\label{q_edge_apexing}
Let $\mathcal{G}$ be a hereditary class of graphs. If $\mathcal{G}$ has a finite number of forbidden induced subgraphs, then so does its $q$-edge-apex class.
\end{corollary}

Similarly, we obtain the following from Theorem \ref{edge_add_bound}. 

\begin{corollary}
\label{p_edge_add}
Let $\mathcal{G}$ be a hereditary class of graphs. If $\mathcal{G}$ has a finite number of forbidden induced subgraphs, then so does its $p$-edge-add class.
\end{corollary}

\section{Forbidden induced subgraphs for Edge-Add chordal graphs}

A graph $G$ is an {\bf edge-add chordal graph} if $G$ is in $\Ga$, where $\mathcal{G}$ is the class of chordal graphs. Recall that a graph is chordal if it contains no induced cycle of length four or greater. 

Because the class of chordal graphs does not have a finite list of forbidden induced subgraphs, we cannot apply the general finiteness results. Therefore, before proving our main characterization in Theorem \ref{edge_add_chordal}, we first establish three general preliminary results regarding the forbidden induced subgraphs of $\Ga$ for an arbitrary hereditary class $\mathcal{G}$.

\begin{lemma}
\label{unique_FIS_lemma_disjoint_case}
Let $\mathcal{G}$ be a hereditary class of graphs and let $G$ be a forbidden induced subgraph for $\Ga$. If $G$ contains two forbidden induced subgraphs $H_1$ and $H_2$ for $\mathcal{G}$ such that $V(H_1) \cap V(H_2)$ is empty or $G[V(H_1) \cap V(H_2)]$ is a clique, then $V(G) = V(H_1) \cup V(H_2)$.
\end{lemma}

\begin{proof}
 Suppose not, and let $x$ be a vertex in $V(G) - (V(H_1) \cup V(H_2))$. Because $V(H_1) \cap V(H_2)$ is a clique or empty, $H_1$ and $H_2$ share no non-edges. Consequently, no single edge addition can simultaneously destroy both $H_1$ and $H_2$. Therefore, $G-x \notin \mathcal{G}^{add}$, which contradicts the minimality of $G$. It follows that $V(G) = V(H_1) \cup V(H_2)$. 
\end{proof}

\begin{lemma}
\label{unique_FIS_lemma}
Let $\mathcal{G}$ be a hereditary class of graphs and let $G$ be a forbidden induced subgraph for $\Ga$. If a forbidden induced subgraph $H$ for $\mathcal{G}$ is contained in $G$, then for every non-edge $h$ of $H$, we have a subset $F_h$ of $V(G)$ such that $(G+h)[F_h]$ is a forbidden induced subgraph for $\mathcal{G}$. Moreover, $V(G) = V(H) \cup \bigcup_{h \in E(\overline{H})} F_h$. 
\end{lemma}

\begin{proof}
Observe that for every non-edge $h$ of $H$, the graph $G+h$ is not in $\mathcal{G}$ so we have a subset $F_h$ of $V(G)$ such that $(G+h)[F_h]$ is a forbidden induced subgraph for $\mathcal{G}$. If there is a vertex $x$ in $V(G)- (V(H) \cup \bigcup F_h )$, then $G-x$ is not in $\Ga$, a contradiction. Therefore, $V(G) = V(H) \cup \bigcup_{h \in E(\overline{H})} F_h$.
\end{proof}

\begin{proposition}
\label{unique_FIS}
Let $\mathcal{G}$ be a hereditary class of graphs, and let $\mathcal{H}$ be a finite subset of the set of forbidden induced subgraphs for $\mathcal{G}$. For a graph $H$ in $\mathcal{H}$, let $c$ and $k$ be the number of vertices and the number of non-edges of $H$, respectively and let $m$ denote the maximum number of vertices in a graph in $\mathcal{H}$. If $G$ is a forbidden induced subgraph for $\mathcal{G}^{add}$ containing $H$, and for every non-edge $h \in E(\overline{H})$ the induced subgraph $(G+h)[F_h]$ belongs to $\mathcal{H}$, then $|V(G)| \leq \max{\{c + m, c + k(m-2)\}}$.
\end{proposition}

\begin{proof}
First note that if $G$ contains two forbidden induced subgraphs $H$ and $H'$ for $\mathcal{G}$ such that $V(H) \cap V(H')$ is empty or $G[V(H) \cap V(H')]$ is complete, then by Lemma \ref{unique_FIS_lemma_disjoint_case},  $|V(G)| \leq c+m$.

By Lemma \ref{unique_FIS_lemma}, it follows that $V(G) = V(H) \cup \bigcup_{h \in E(\overline{H})} F_h$. Observe that, if $|V(H) \cap F_h| \leq 1$, then $G[V(H) \cap F_h]$ is a clique and thus $|V(G)| \leq c+m$. So for every non-edge $h$ of $H$, we have $|V(H) \cap F_h| \geq 2$. Therefore, for every non-edge $h$ of $H$, the set $F_h$ contributes at most $m-2$ vertices outside of $V(H)$. Since $H$ has $k$ non-edges, the total number of vertices in $G$ is bounded by $|V(H)| + k(m-2) = c + k(m-2)$. Combining this with the disjoint case, we obtain $|V(G)| \le \max\{c+m, c+k(m-2)\}$.
\end{proof}

\begin{theorem}
\label{edge_add_chordal}
Let $G$ be a forbidden induced subgraph for the class of edge-add chordal graphs. Then $G$ is a cycle of length at least five, or $6 \leq |V(G)| \leq 8$, and $G$ is isomorphic to one of the graphs in Figures \ref{chordal_6_7} and \ref{chordal_8}.
\end{theorem}

\begin{proof}
Since $G$ is not a chordal graph, it must contain a cycle of length at least four. Observe that a cycle of length at least five is not an edge-add chordal graph. It follows that if $G$ contains a cycle $C$ of length at least five, then $G=C$, and the result follows so assume that $G$ contains no cycle of length greater than $4$. Therefore $G$ contains an induced cycle $C$ of length four, say $C = abcda$. 

Note that, by Lemma \ref{unique_FIS_lemma}, corresponding to the non-edges $ac$ and $bd$ of $C$, we have subsets $C_{ac}$ and $C_{bd}$ of $V(G)$ of size at least four such that $(G+ac)[C_{ac}]$ and $(G+bd)[C_{bd}]$ are cycles. By Lemma \ref{unique_FIS_lemma}, it follows that $V(G) = \{a,b,c,d\} \cup C_{ac} \cup C_{bd}$. Note that if $G[C_{ac}]$ is a cycle, then $C_{ac}$ has size four. Moreover, if $|C_{ac} \cap V(C)| < 2$, then by Lemma \ref{unique_FIS_lemma_disjoint_case}, we obtain $|V(G)| \leq 8$. So we may assume that $|C_{ac} \cap V(C)| \geq 2$. Since the same holds for $C_{bd}$, it follows that $|V(G)| \leq 8$. Therefore we may assume that both $G[C_{ac}]$ and $G[C_{bd}]$ are not cycles.

It follows that $G[C_{ac}]$ and $G[C_{bd}]$ are paths $P_{ac}$ and $P_{bd}$, respectively of length at least three, say $P_{ac}= ax_1x_2 \ldots c$ and $P_{bd}= by_1y_2 \ldots d$. We call the vertices of a path other than the endpoints \textit{internal}. Note that $b$ is adjacent to all the internal vertices of $P_{ac}$. Suppose not. Then we obtain an induced cycle of length greater than four in $G$, or an induced cycle $D$ of length four in $G$. Because no internal vertex of $P_{ac}$ connects to both $a$ and $c$, the intersection $V(C) \cap V(D)$ is restricted to at most $\{b\}$, $\{a,b\}$, or $\{b,c\}$, all of which induce cliques in $G$. The result then follows by Lemma \ref{unique_FIS_lemma_disjoint_case}. By symmetry, $d$ is also adjacent to all the internal vertices of $P_{ac}$. By a similar argument, both $a$ and $c$ are adjacent to all the internal vertices of $P_{bd}$. 

Observe that if both paths $P_{ac}$ and $P_{bd}$ have exactly two internal vertices, then $|V(G)| \leq 8$, so assume $P_{ac}$ has at least three internal vertices, say $x_1, x_2$, and $x_3$. Because $b$ and $d$ are adjacent to all internal vertices of $P_{ac}$, the sets $\{b, x_1, d, x_3\}$ and $\{a, b, x_3, d\}$ both induce 4-cycles in $G-c$. To destroy these multiple induced 4-cycles with a single edge addition, the added edge must be their only common non-edge, namely $bd$. However, because $P_{bd}$ has length at least three, it cannot be the path $bcd$, that is, $P_{bd}$ does not contain $c$ and so remains a path in $G-c$. Adding the edge $bd$ closes $P_{bd}$ into an induced cycle of length at least four. Consequently, no single edge addition can make $G-c$ chordal so $G-c$ is not an edge-add chordal graph. This is a contradiction to the minimality of $G$.

Because our structural proofs establish that any minimal forbidden induced subgraph for this class has at most 8 vertices, the question of completeness is reduced to a finite search. To verify the exact structures of these obstructions, we implemented the following exhaustive search algorithm using SageMath \cite{sage} and nauty \cite{nauty} and list them in Figures \ref{chordal_6_7} and \ref{chordal_8}. 
\end{proof}

\begin{algorithm}
\label{pseudocode2}
\caption{Exhaustive verification of obstructions for $\Ga$}
\begin{algorithmic}
\STATE Set FinalList $\leftarrow \emptyset$
\STATE Generate all graphs of order $n \in \{5,6,7,8\}$ using nauty geng \cite{nauty} and store in an iterator $L$

\FOR{$g$ in $L$ such that $g$ is not in $\mathcal{G}$}
    \STATE Set $i \leftarrow 0$, $j \leftarrow 0$

    \FOR{$e$ in $E(\overline{g})$}
        \STATE $h = g + e$
        \IF{$h$ is not in $\mathcal{G}$}
            \STATE $i \leftarrow i+1$
        \ENDIF
    \ENDFOR

    \FOR{$v$ in $V(g)$}
        \STATE $k = g - v$
        \IF{$k$ is in $\mathcal{G}$ or $k+e$ is in $\mathcal{G}$ for some non-edge $e$ of $k$}
            \STATE $j \leftarrow j+1$
        \ENDIF
    \ENDFOR

\IF{$i$ equals $|E(g)|$ and $j$ equals $|V(g)|$}

\STATE Add $g$ to FinalList

\ENDIF

\ENDFOR

\STATE Return FinalList

\end{algorithmic}
\end{algorithm}

\begin{theorem}
\label{p_edge_add_chordal}
Let $\mathcal{F}$ be the set of all forbidden induced subgraphs for the class of $p$-edge-add chordal graphs that are not cycles. Then $\mathcal{F}$ is finite. 
\end{theorem}

\begin{proof}
We proceed by induction on $p$. By Theorem \ref{edge_add_chordal}, the result holds when $p \in \{0,1\}$ so assume $p \geq 2$. Observe that the class of $p$-edge-add chordal graphs is precisely the edge-add class of $(p-1)$-edge-add chordal graphs. Assume the result holds for $(p-1)$-edge-add chordal graphs, and let $\mathcal{H}$ be its finite set of non-cycle forbidden induced subgraphs. 

Let $G$ be a non-cycle forbidden induced subgraph for the class of $p$-edge-add chordal graphs. Note that an induced cycle of length at least $p+4$ is not a $p$-edge-add chordal graph since it requires at least $p+1$ edge additions to become chordal. It follows that if $G$ contains an induced cycle $C$ of length at least $p+4$, then $G=C$. Therefore we may assume that $G$ contains no induced cycle of length $\ge p+4$. Let $K = p+4$.

We note the following. 

\begin{sublemma}
\label{diameter}
Let $x,y$ be a pair of vertices of $G$ at a distance $d$. Then any induced path between $x$ and $y$ has length at most $d(K-2)$. 
\end{sublemma}

Because $G$ is not $(p-1)$-edge-add chordal, it must contain a forbidden induced subgraph $H$ for the class of $(p-1)$-edge-add chordal graphs. Since an induced cycle of length at most $p+2$ is a $(p-1)$-edge-add chordal graph, it follows that $H$ is in $\mathcal{H}$, or $H$ is a cycle of length $p+3$. In either case, the number of vertices in $H$ is  bounded by a constant $m$. 

By Lemma \ref{unique_FIS_lemma}, for every non-edge $h \in E(\overline{H})$, there exists a vertex subset $F_h \subseteq V(G)$ such that $(G+h)[F_h]$ is a forbidden induced subgraph for $(p-1)$-edge-add chordal graphs, and $V(G) = V(H) \cup \bigcup_{h \in E(\overline{H})} F_h$. 

Note that if $H$ is disconnected and $h = uv$ is a non-edge with endpoints in distinct connected components of $H$, then adding $h$ does not create any new cycles, nor does it triangulate any existing cycles in $H$. Consequently, $(G+h)[V(H)]$ remains a forbidden induced subgraph for $(p-1)$-edge-add chordal graphs, that is we can simply take $F_h \subseteq V(H)$. Therefore, to account for the vertices of $G$ outside of $H$, we only need to focus on the subsets $F_h$ corresponding to non-edges $h$ whose endpoints lie within the same connected component of $H$. Let $E^*(\overline{H})$ denote the set of non-edges of $H$ whose endpoints belong to the same connected component. We can then refine our vertex union to:
$$V(G) = V(H) \cup \bigcup_{h \in E^*(\overline{H})} F_h.$$

Because $|V(H)| \le m$, the number of non-edges $h$ in $E^*(\overline{H})$ is  bounded. To prove $\mathcal{F}$ is finite, it suffices to show that the size of every set $F_h$ is bounded by a universal constant. Fix a non-edge $h=xy$ in $E^*(\overline{H})$. The subgraph $(G+h)[F_h]$ must either belong to the finite family $\mathcal{H}$ or be an induced cycle. 

If $(G+h)[F_h] \in \mathcal{H}$ or is an induced cycle of length at most $K-1$, the size of $F_h$ is bounded by a constant dependent only on $m$ and $K$. So we may assume that $(G+h)[F_h]$ is an induced cycle of length at least $K$. Note that this cycle must utilize the edge $h$, otherwise it would be an induced cycle in $G$ of length $\ge K$, a contradiction. Therefore, $G[F_h]$ is an induced path in $G$ connecting $x$ and $y$. 

Since $x$ and $y$ belong to the same connected component of $H$, their distance in $H$ is at most $m-1$. Consequently, their distance $d$ in $G$ is bounded by a fixed constant $D \le m-1$. Furthermore, because $G$ lacks induced cycles of length $\ge K$, by Sublemma \ref{diameter}, the length of the induced path $G[F_h]$ between $x$ and $y$ is  bounded by $D(K-2)$. 

Thus, the size of every $F_h$ is bounded by a universal constant $L$ dependent only on the fixed constants $m$ and $K$. Because $V(G)$ is the union of $V(H)$ and a bounded number of sets $F_h$, the total number of vertices in $G$ is  bounded by a constant, proving $\mathcal{F}$ is finite.
\end{proof}

A graph is \textbf{weakly chordal} if every induced cycle in $G$ and $\overline{G}$ has length at most $4$. It is well known that weakly chordal graphs are perfect \cite{RH}.

\begin{theorem}
\label{edge_add_perfect}
Let $\mathcal{F}$ be the set of all forbidden induced subgraphs for the class of edge-add perfect graphs that are not cycles or their complements. Then $\mathcal{F}$ is infinite. 
\end{theorem}

\begin{proof}
The strong perfect graph theorem \cite{CSRT} asserts that the forbidden induced subgraphs for the class of perfect graphs are the odd cycles of length at least $5$ and their complements. But all these graphs are edge-add-perfect. (Adding an edge between two neighbors of a vertex in an odd cycle of length at least $5$ gives a perfect graph. Adding an edge between two non-adjacent vertices in the complement of an odd cycle of length at least $5$ is the complement of a path graph, which is perfect.) Notice that the disjoint union of two graphs from the set of  odd cycles of length at least $5$ and the complements of an odd cycle of length at least $5$ is a forbidden induced subgraph for the class of edge-add perfect graphs. This proves the theorem. 
\end{proof}

The situation for weakly chordal graphs is not clear. We offer the following conjecture: 

\begin{conjecture}
\label{edge_add_weaklychordal}
Let $\mathcal{F}$ be the set of all forbidden induced subgraphs for the class of edge-add weakly chordal graphs that are not cycles or their complements. Then $\mathcal{F}$ is finite. 
\end{conjecture}

\begin{figure}[htbp]
\centering
\begin{minipage}{.17\linewidth}
  \includegraphics[scale=0.25]{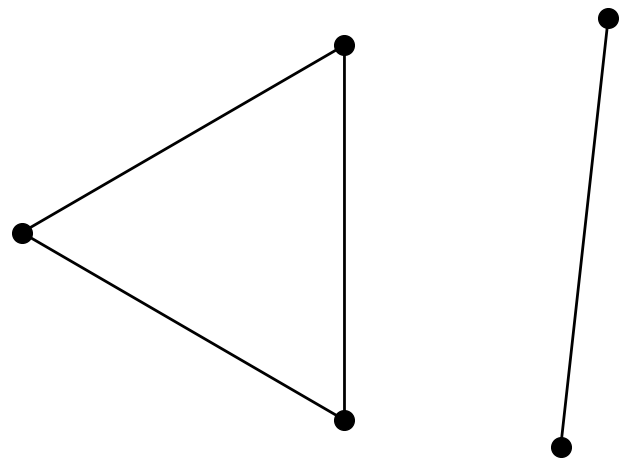}
\end{minipage} \hspace{.30\linewidth}
\begin{minipage}{.17\linewidth}
  \includegraphics[scale=0.25]{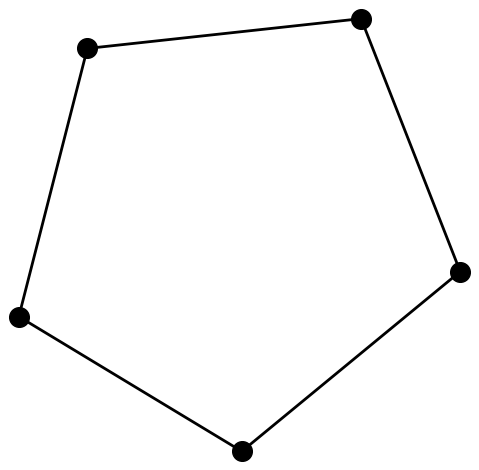}
\end{minipage} \hspace{.01\linewidth}
\caption{The $5$-vertex forbidden induced subgraphs for edge-add split graphs.}
\label{split_5}
\end{figure}

\begin{figure}[htbp]
\centering
\begin{minipage}{.17\linewidth}
  \includegraphics[scale=0.2]{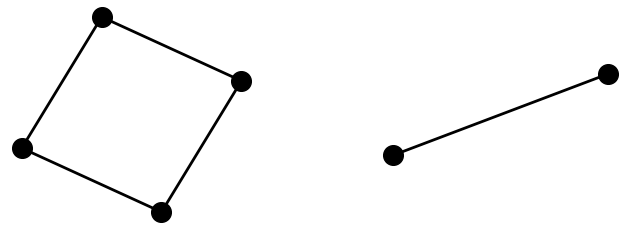}
\end{minipage}
\hspace{.09\linewidth}
\begin{minipage}{.17\linewidth}
\includegraphics[scale=0.2]{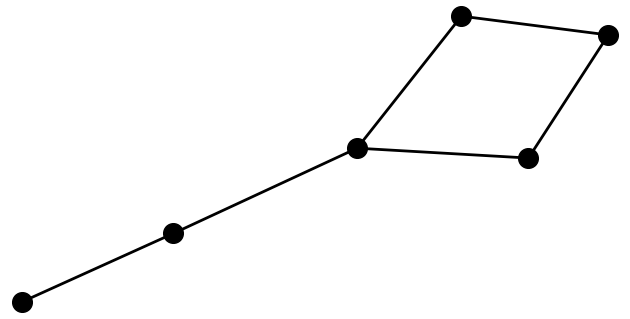}
\end{minipage}
\hspace{.09\linewidth}
\begin{minipage}{.17\linewidth}
\includegraphics[scale=0.2]{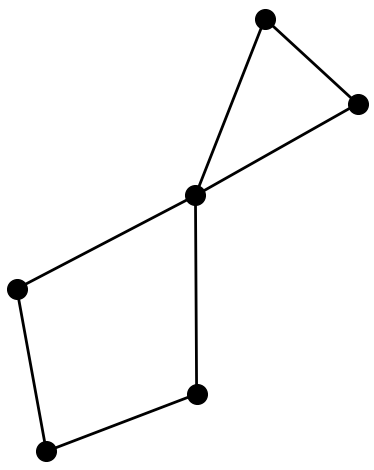}
\end{minipage}
\hspace{.05\linewidth}
\begin{minipage}{.17\linewidth}
  \includegraphics[scale=0.2]{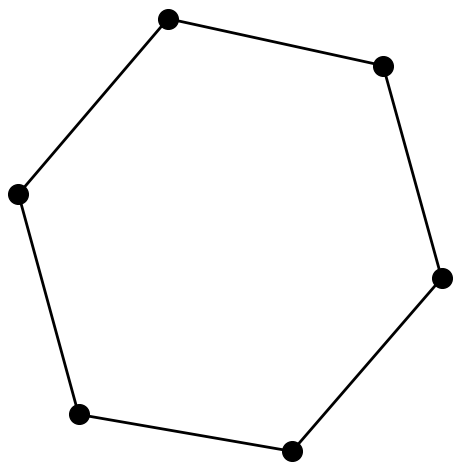}
\end{minipage}

\begin{minipage}{.17\linewidth}
  \includegraphics[scale=0.2]{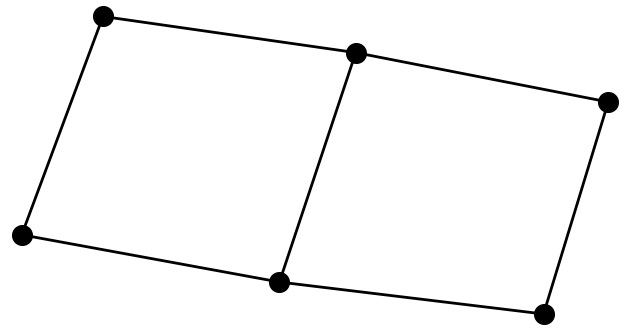}
\end{minipage}
\hspace{.1\linewidth}
\begin{minipage}{.17\linewidth}
\includegraphics[scale=0.2]{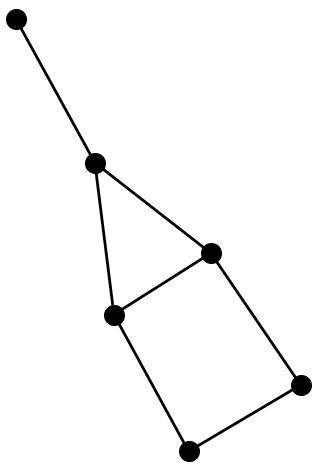}
\end{minipage}
\hspace{.05\linewidth}
\begin{minipage}{.17\linewidth}
\includegraphics[scale=0.2]{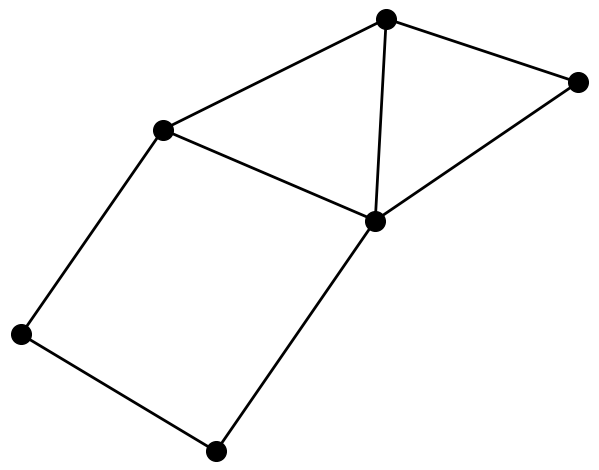}
\end{minipage}
\hspace{.05\linewidth}
\begin{minipage}{.17\linewidth}
  \includegraphics[scale=0.2]{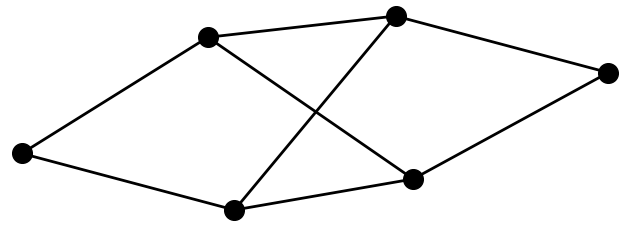}
\end{minipage}
\hspace{.05\linewidth}

\begin{minipage}{.17\linewidth}
  \includegraphics[scale=0.2]{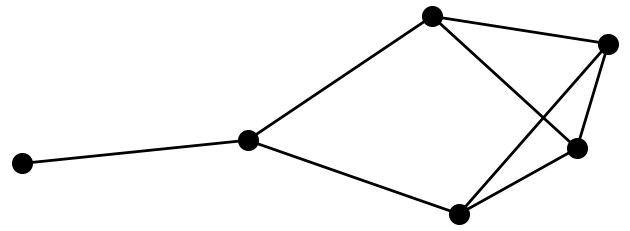}
\end{minipage}
\hspace{.1\linewidth}
\begin{minipage}{.17\linewidth}
\includegraphics[scale=0.2]{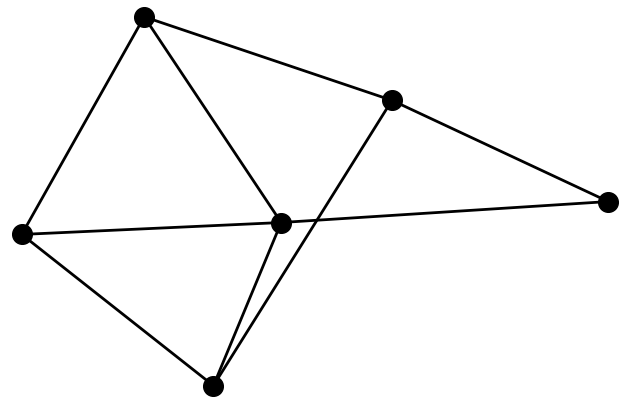}
\end{minipage}
\hspace{.09\linewidth}
\begin{minipage}{.17\linewidth}
\includegraphics[scale=0.2]{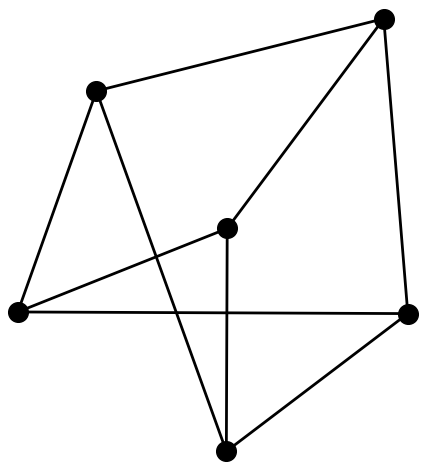}
\end{minipage}
\hspace{.05\linewidth}
\begin{minipage}{.17\linewidth}
  \includegraphics[scale=0.2]{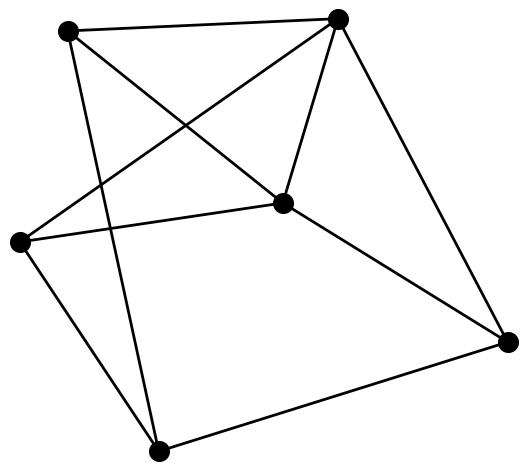}
\end{minipage}
\hspace{.05\linewidth}

\begin{minipage}{.17\linewidth}
  \includegraphics[scale=0.2]{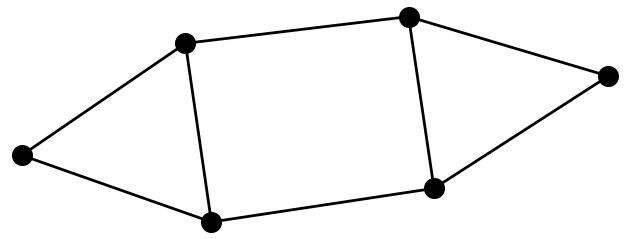}
\end{minipage}
\hspace{.12\linewidth}
\begin{minipage}{.17\linewidth}
\includegraphics[scale=0.2]{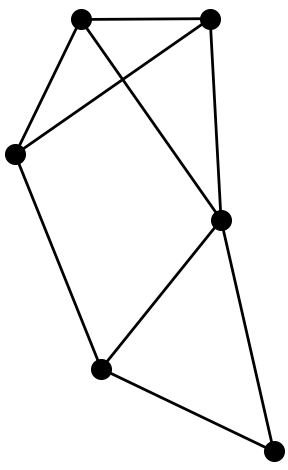}
\end{minipage}
\hspace{.05\linewidth}
\begin{minipage}{.17\linewidth}
\includegraphics[scale=0.2]{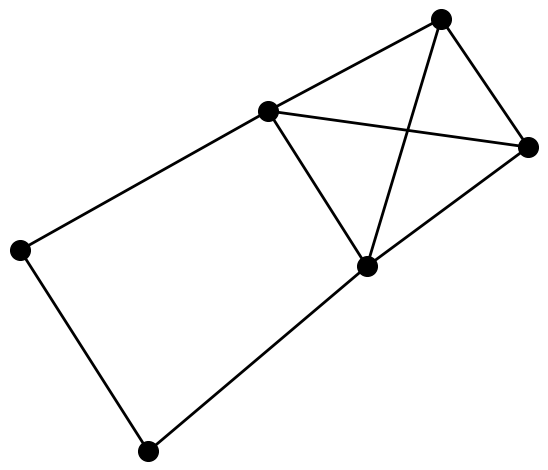}
\end{minipage}
\hspace{.05\linewidth}
\begin{minipage}{.17\linewidth}
  \includegraphics[scale=0.2]{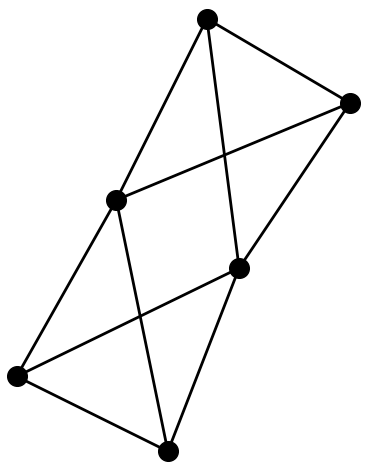}
\end{minipage}
\hspace{.05\linewidth}

\begin{minipage}{.17\linewidth}
  \includegraphics[scale=0.2]{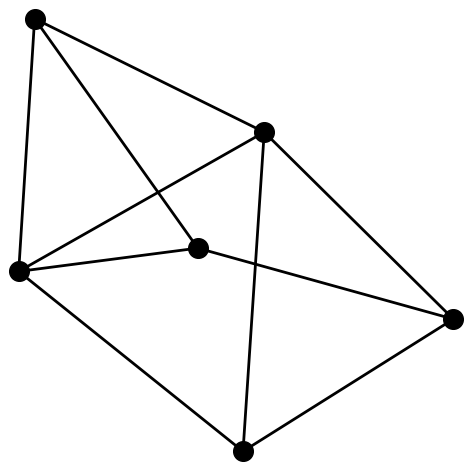}
\end{minipage}
\hspace{.12\linewidth}
\begin{minipage}{.17\linewidth}
\includegraphics[scale=0.2]{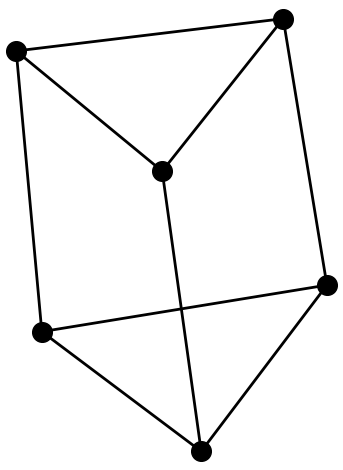}
\end{minipage}
\hspace{.05\linewidth}
\begin{minipage}{.17\linewidth}
\includegraphics[scale=0.2]{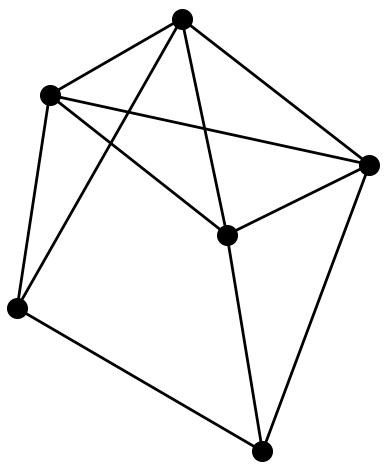}
\end{minipage}
\hspace{.05\linewidth}
\begin{minipage}{.17\linewidth}
  \includegraphics[scale=0.2]{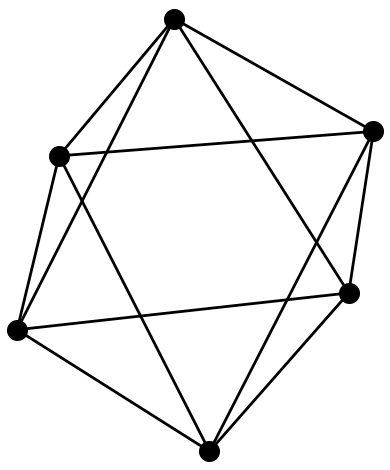}
\end{minipage}
\hspace{.05\linewidth}

\begin{minipage}{.17\linewidth}
  \includegraphics[scale=0.2]{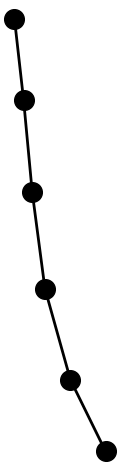}
\end{minipage}
\hspace{.05\linewidth}
\begin{minipage}{.17\linewidth}
\includegraphics[scale=0.2]{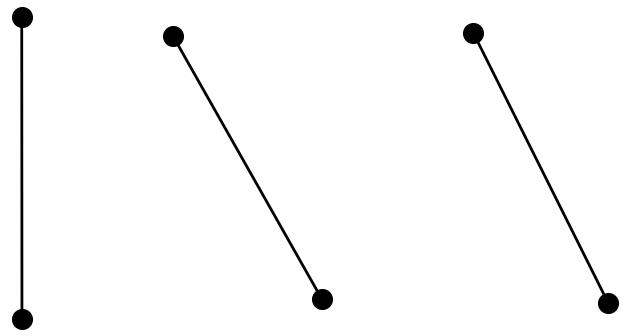}
\end{minipage}
\hspace{.15\linewidth}
\begin{minipage}{.17\linewidth}
\includegraphics[scale=0.2]{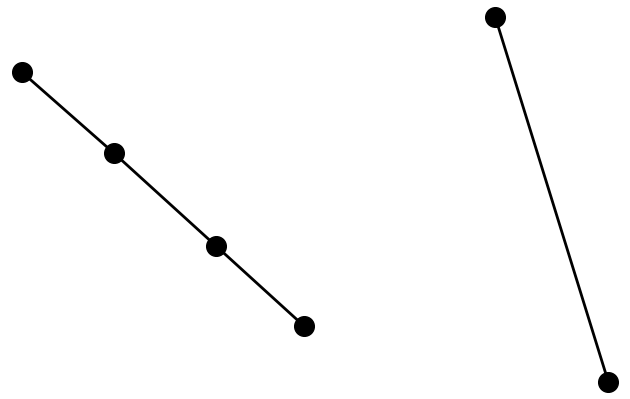}
\end{minipage}
\hspace{.05\linewidth}

\caption{The $6$-vertex forbidden induced subgraphs for edge-add split graphs.}

\label{split_6}

\end{figure}

\begin{figure}[htbp]
\centering

\begin{minipage}{.17\linewidth}
\includegraphics[scale=0.25]{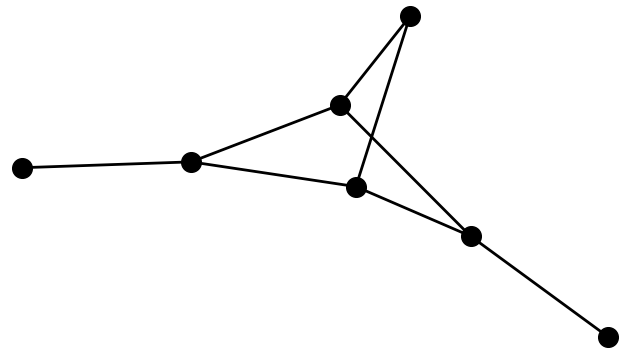}
\end{minipage}\hspace{.15\linewidth}
\begin{minipage}{.17\linewidth}
  \includegraphics[scale=0.25]{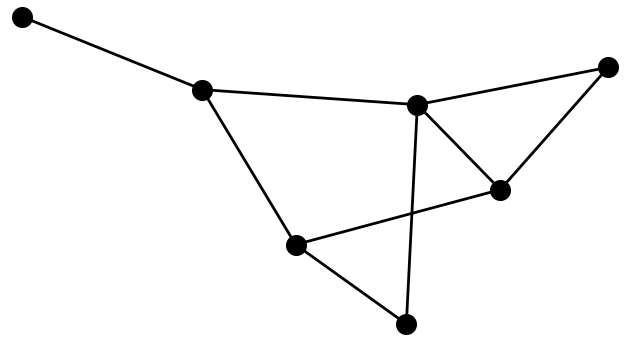}
\end{minipage} \hspace{.20\linewidth}
\begin{minipage}{.17\linewidth}
  \includegraphics[scale=0.25]{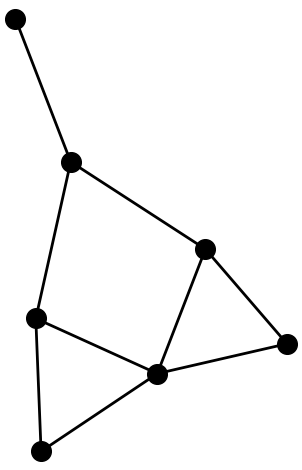}
\end{minipage}

\begin{minipage}{.17\linewidth}
\includegraphics[scale=0.25]{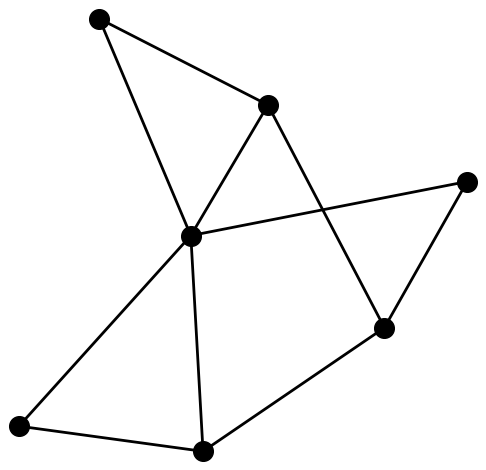}
\end{minipage}\hspace{.21\linewidth}
\begin{minipage}{.17\linewidth}
  \includegraphics[scale=0.25]{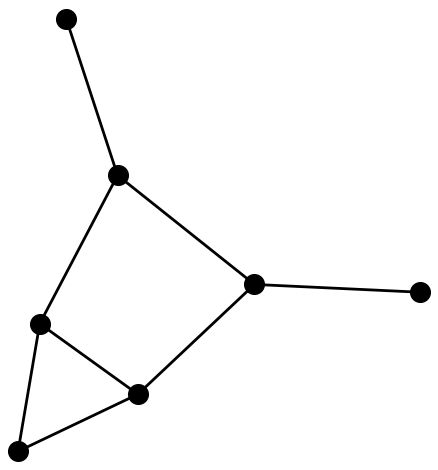}
\end{minipage} \hspace{.15\linewidth}
\begin{minipage}{.17\linewidth}
  \includegraphics[scale=0.25]{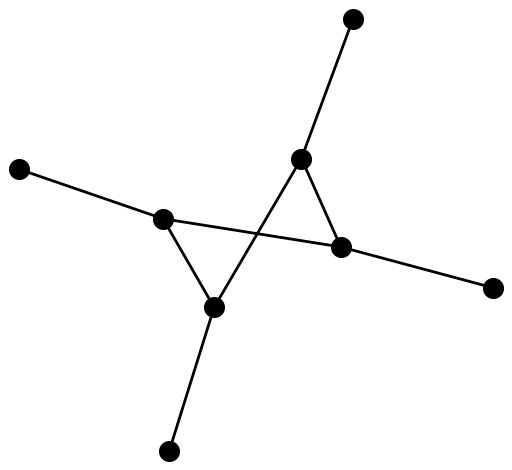}
\end{minipage}

\caption{The $7$-vertex and $8$-vertex forbidden induced subgraphs for edge-add split graphs.}

\label{split_7_8}

\end{figure}

\begin{figure}[htbp]
\centering
\begin{minipage}{.17\linewidth}
  \includegraphics[scale=0.2]{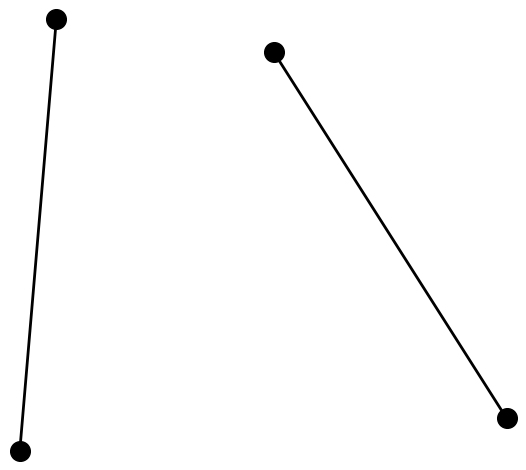}
\end{minipage}
\hspace{.08\linewidth}
\begin{minipage}{.17\linewidth}
\includegraphics[scale=0.2]{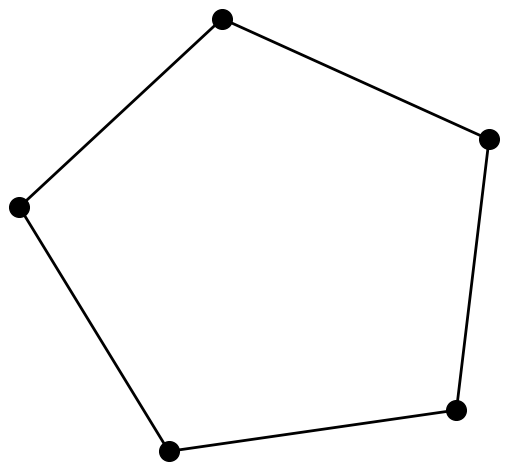}
\end{minipage}
\hspace{.05\linewidth}
\begin{minipage}{.17\linewidth}
\includegraphics[scale=0.2]{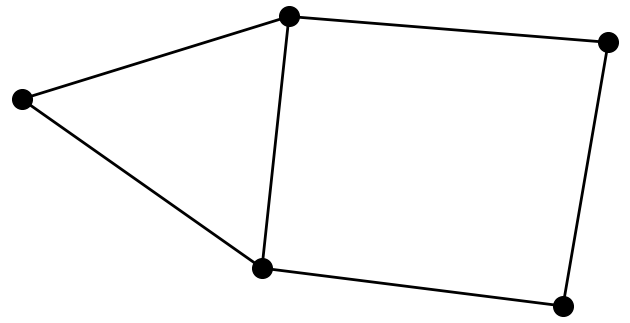}
\end{minipage}
\hspace{.1\linewidth}
\begin{minipage}{.17\linewidth}
  \includegraphics[scale=0.2]{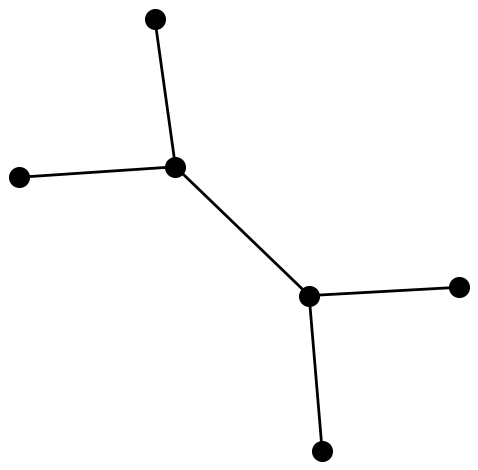}
\end{minipage}
\hspace{.05\linewidth}

\begin{minipage}{.17\linewidth}
  \includegraphics[scale=0.2]{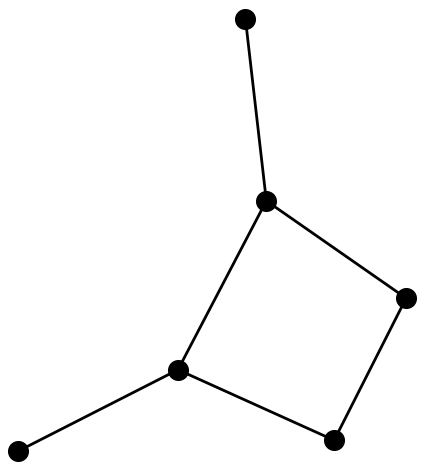}
\end{minipage} \hspace{.08\linewidth}
\begin{minipage}{.17\linewidth}
  \includegraphics[scale=0.25]{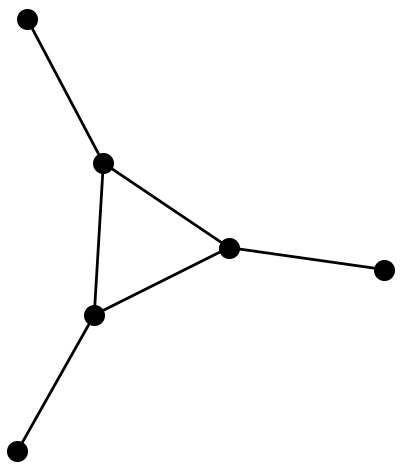}
\end{minipage}
\hspace{.05\linewidth}
\begin{minipage}{.17\linewidth}
\includegraphics[scale=0.25]{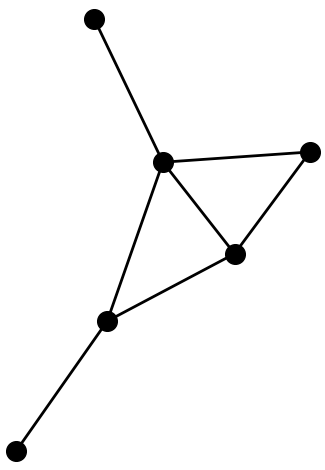}
\end{minipage}
\hspace{.11\linewidth}
\begin{minipage}{.17\linewidth}
\includegraphics[scale=0.25]{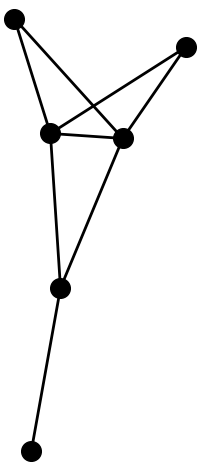}
\end{minipage}

\begin{minipage}{.17\linewidth}
  \includegraphics[scale=0.2]{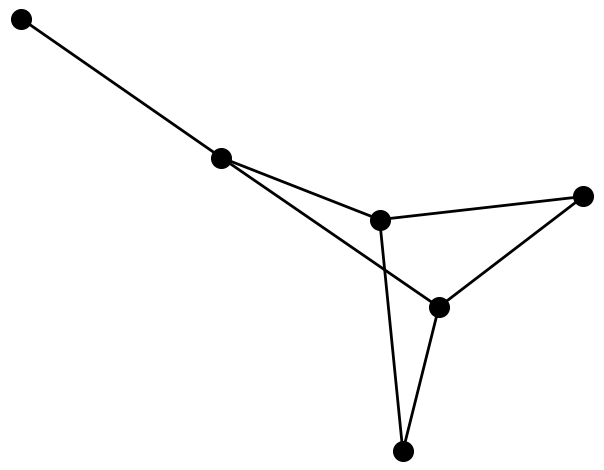}
\end{minipage}
\hspace{.1\linewidth}
\begin{minipage}{.17\linewidth}
  \includegraphics[scale=0.25]{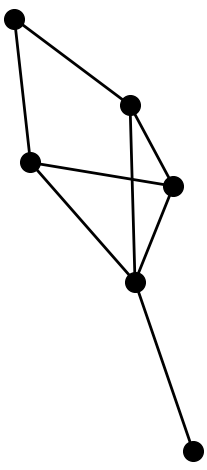}
\end{minipage} \hspace{.01\linewidth}
\begin{minipage}{.17\linewidth}
  \includegraphics[scale=0.2]{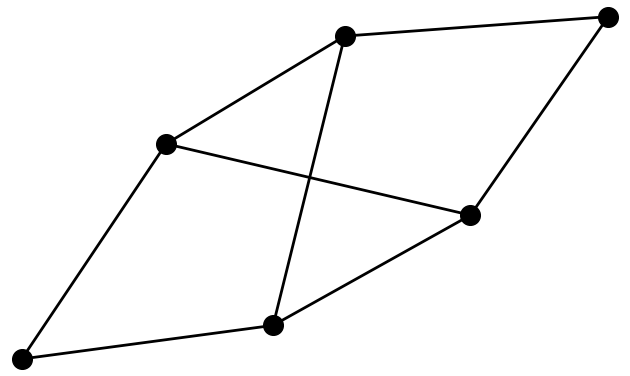}
\end{minipage} \hspace{.08\linewidth}
\begin{minipage}{.17\linewidth}
\includegraphics[scale=0.2]{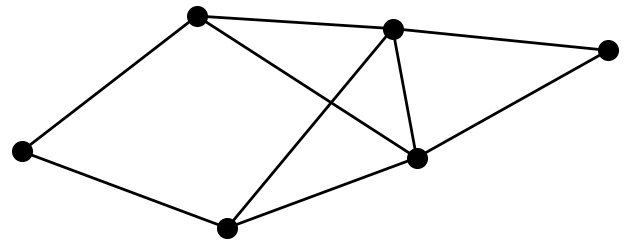}
\end{minipage}
\hspace{.05\linewidth}

\begin{minipage}{.17\linewidth}
\includegraphics[scale=0.2]{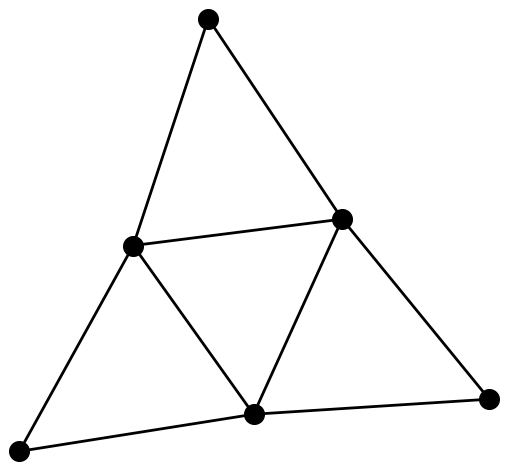}
\end{minipage}\hspace{.08\linewidth}
\begin{minipage}{.17\linewidth}
  \includegraphics[scale=0.22]{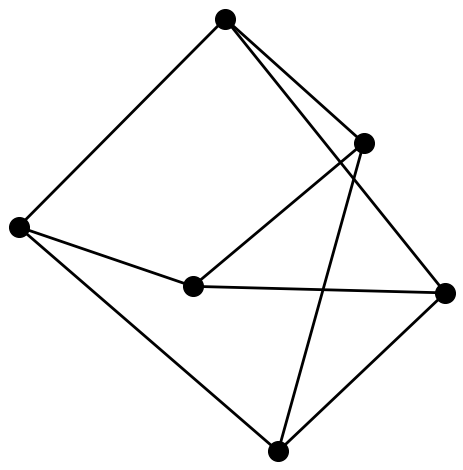}
\end{minipage}
\hspace{.07\linewidth}
\begin{minipage}{.17\linewidth}
  \includegraphics[scale=0.25]{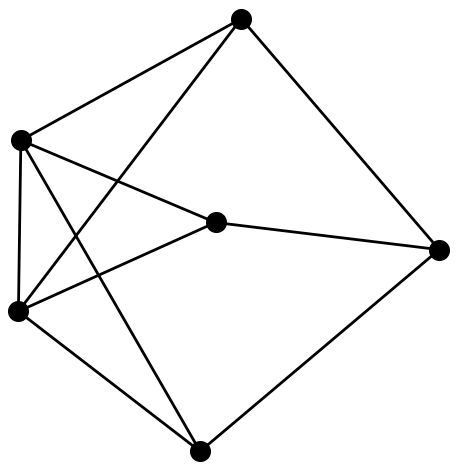}
\end{minipage} \hspace{.11\linewidth}
\begin{minipage}{.17\linewidth}
  \includegraphics[scale=0.27]{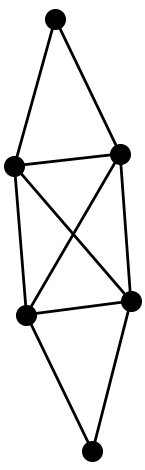}
\end{minipage}
\hspace{.05\linewidth}

\begin{minipage}{.17\linewidth}
\includegraphics[scale=0.2]{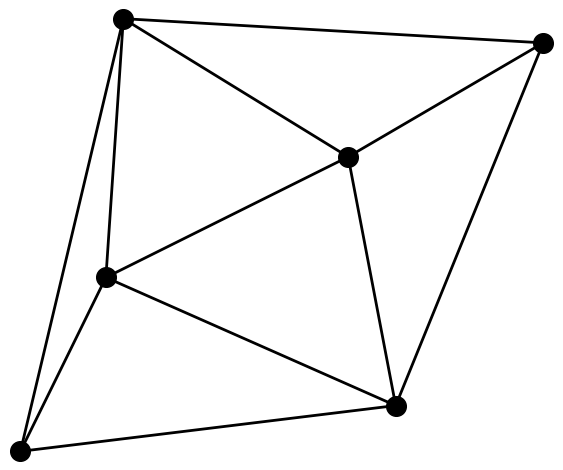}
\end{minipage}\hspace{.07\linewidth}
\begin{minipage}{.17\linewidth}
  \includegraphics[scale=0.2]{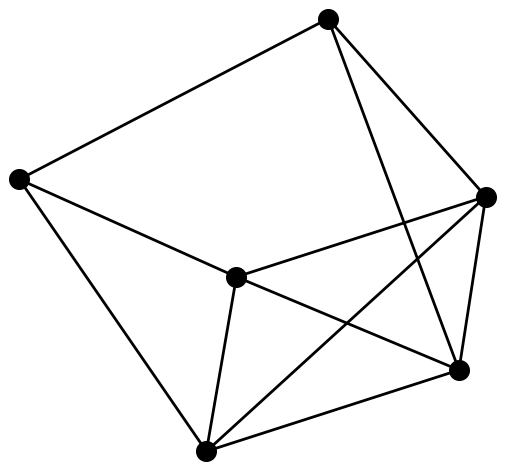}
\end{minipage}
\hspace{.07\linewidth}
\begin{minipage}{.17\linewidth}
  \includegraphics[scale=0.2]{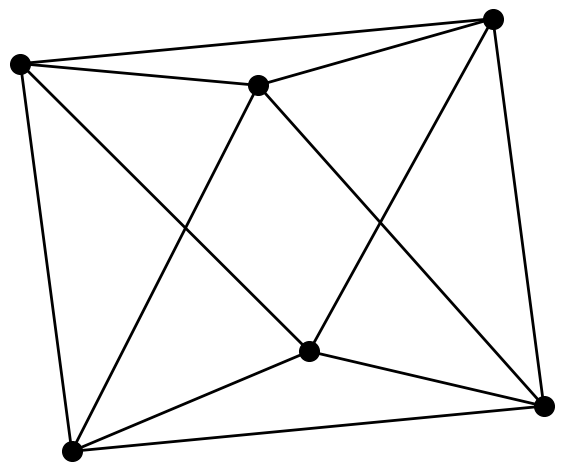}
\end{minipage} \hspace{.10\linewidth}
\begin{minipage}{.17\linewidth}
  \includegraphics[scale=0.27]{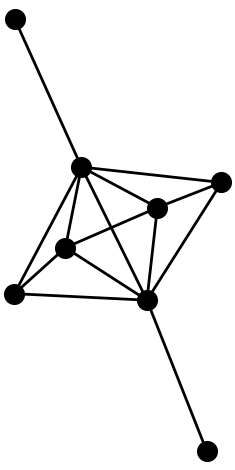}
\end{minipage}
\hspace{.05\linewidth}

\begin{minipage}{.17\linewidth}
\includegraphics[scale=0.30]{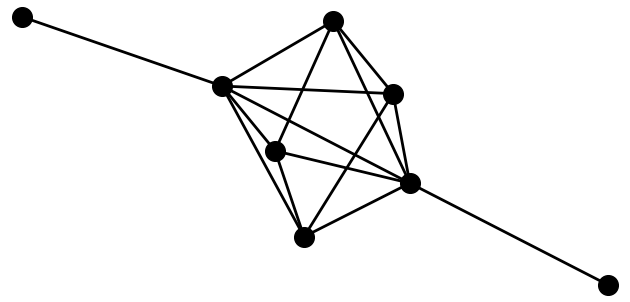}
\end{minipage}
\hspace{.01\linewidth}

\caption{The forbidden induced subgraphs for edge-add threshold graphs.}

\label{threshold}

\end{figure}

\begin{figure}[htbp]
\centering
\begin{minipage}{.17\linewidth}
  \includegraphics[scale=0.2]{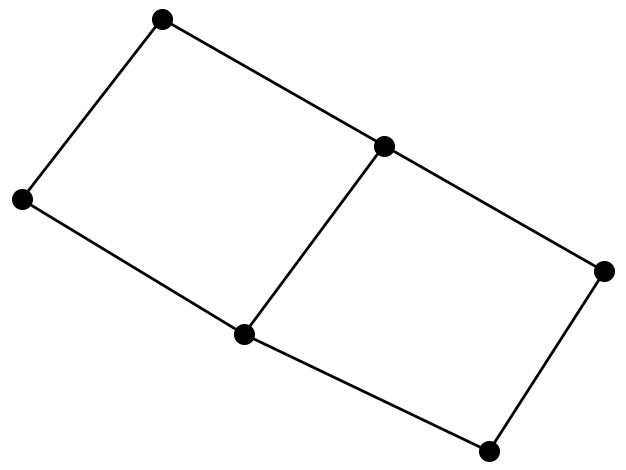}
\end{minipage}
\hspace{.15\linewidth}
\begin{minipage}{.17\linewidth}
\includegraphics[scale=0.2]{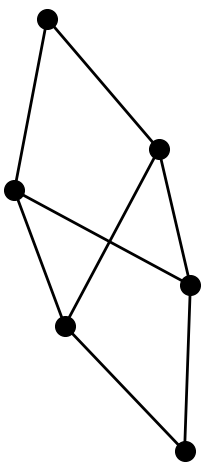}
\end{minipage}
\hspace{.05\linewidth}
\begin{minipage}{.17\linewidth}
\includegraphics[scale=0.2]{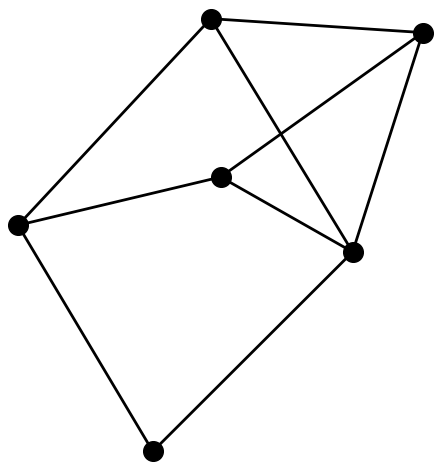}
\end{minipage}
\hspace{.05\linewidth}
\begin{minipage}{.17\linewidth}
  \includegraphics[scale=0.18]{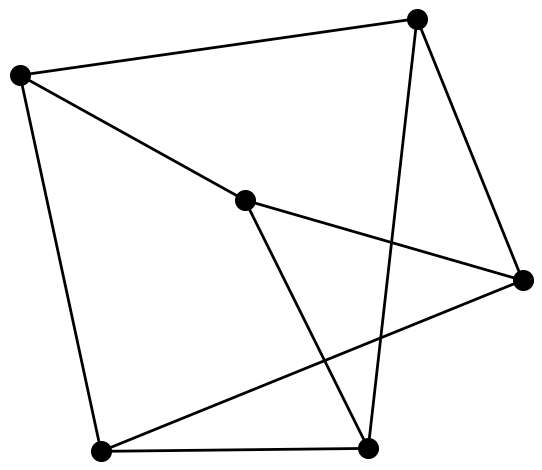}
\end{minipage}

\begin{minipage}{.17\linewidth}
  \includegraphics[scale=0.18]{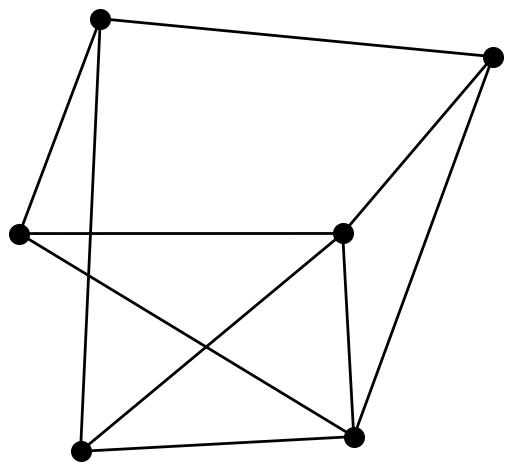}
\end{minipage}
\hspace{.08\linewidth}
\begin{minipage}{.17\linewidth}
\includegraphics[scale=0.18]{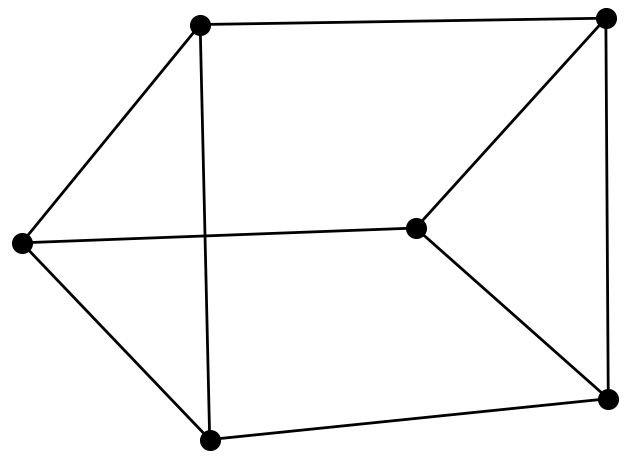}
\end{minipage}
\hspace{.05\linewidth}
\begin{minipage}{.17\linewidth}
\includegraphics[scale=0.18]{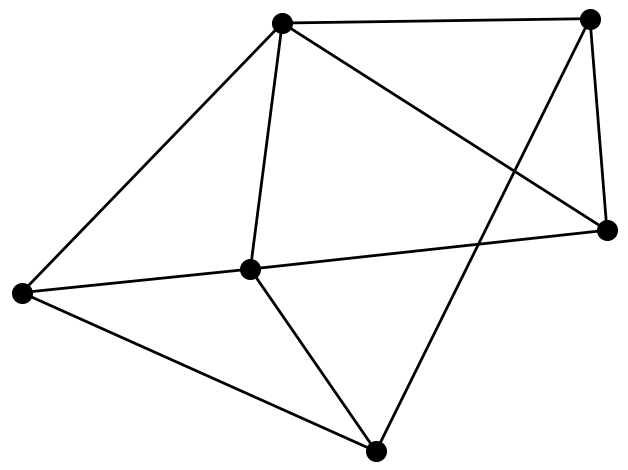}
\end{minipage}
\hspace{.1\linewidth}
\begin{minipage}{.17\linewidth}
  \includegraphics[scale=0.18]{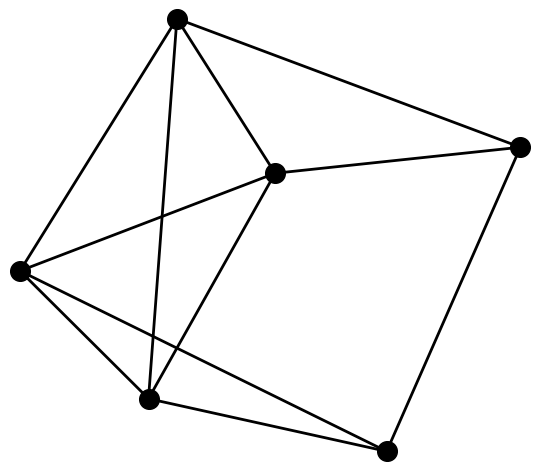}
\end{minipage}
\hspace{.05\linewidth}

\begin{minipage}{.17\linewidth}
  \includegraphics[scale=0.2]{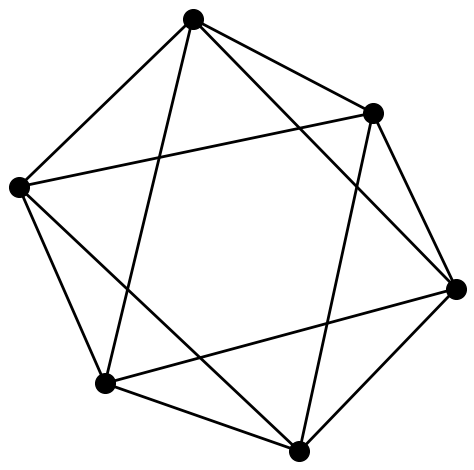}
\end{minipage}
\hspace{.12\linewidth}
\begin{minipage}{.17\linewidth}
\includegraphics[scale=0.2]{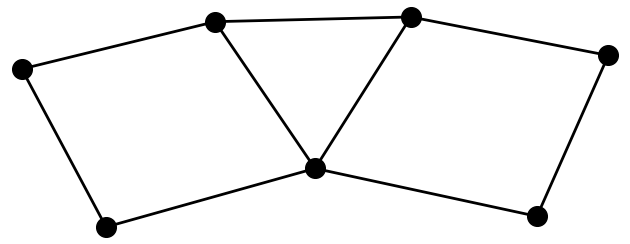}
\end{minipage}
\hspace{.08\linewidth}
\begin{minipage}{.17\linewidth}
\includegraphics[scale=0.2]{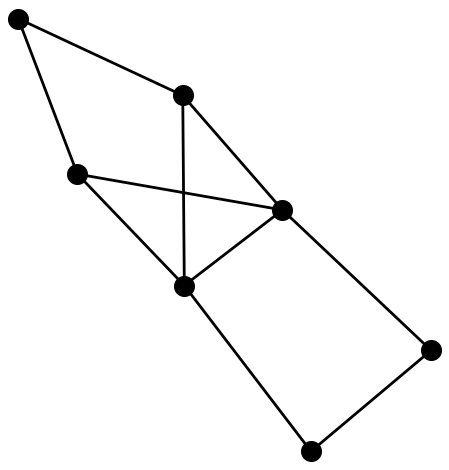}
\end{minipage}
\hspace{.05\linewidth}
\begin{minipage}{.17\linewidth}
  \includegraphics[scale=0.2]{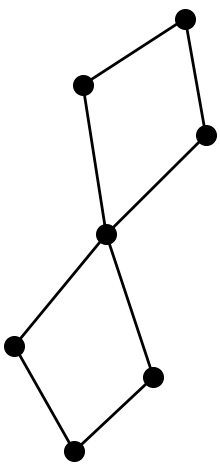}
\end{minipage}
\hspace{.05\linewidth}

\begin{minipage}{.17\linewidth}
  \includegraphics[scale=0.25]{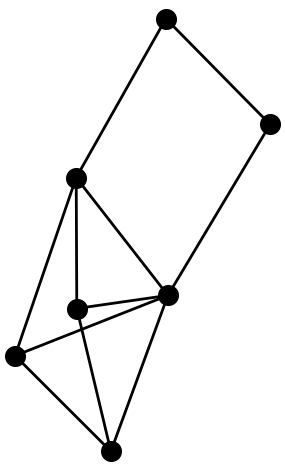}
\end{minipage}
\hspace{.08\linewidth}
\begin{minipage}{.17\linewidth}
\includegraphics[scale=0.23]{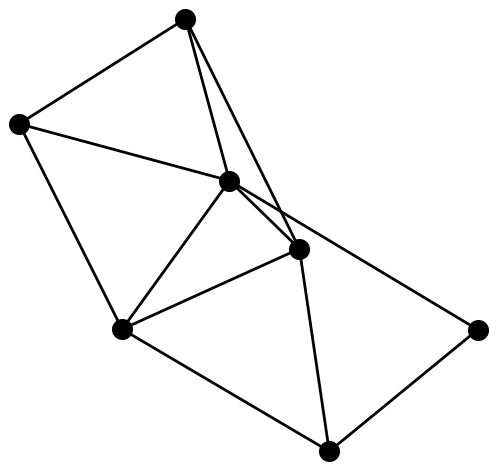}
\end{minipage}
\hspace{.12\linewidth}
\begin{minipage}{.17\linewidth}
\includegraphics[scale=0.23]{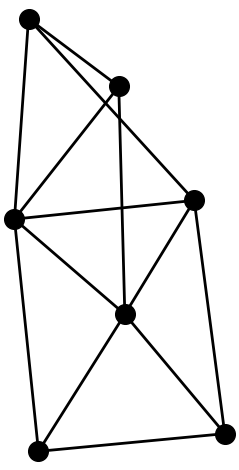}
\end{minipage}
\hspace{.05\linewidth}

\caption{$6$-vertex and $7$-vertex forbidden induced subgraphs for edge-add chordal graphs.}

\label{chordal_6_7}

\end{figure}

\begin{figure}[htbp]
\centering
\begin{minipage}{.17\linewidth}
  \includegraphics[scale=0.18]{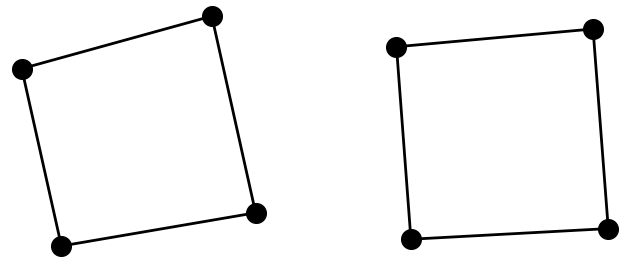}
\end{minipage}
\hspace{.08\linewidth}
\begin{minipage}{.17\linewidth}
\includegraphics[scale=0.2]{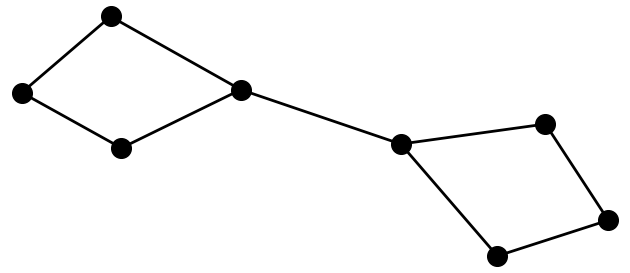}
\end{minipage}
\hspace{.08\linewidth}
\begin{minipage}{.17\linewidth}
\includegraphics[scale=0.2]{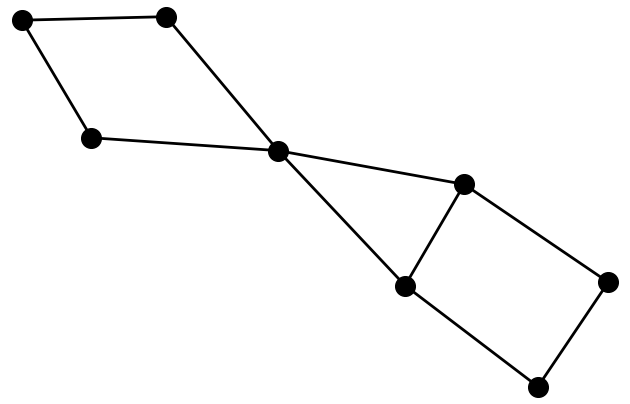}
\end{minipage}
\hspace{.08\linewidth}
\begin{minipage}{.17\linewidth}
  \includegraphics[scale=0.2]{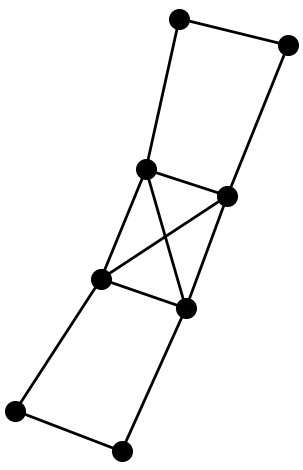}
\end{minipage}

\begin{minipage}{.17\linewidth}
  \includegraphics[scale=0.18]{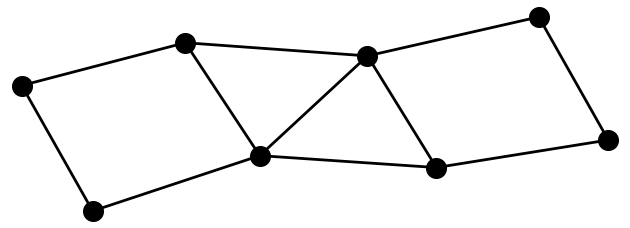}
\end{minipage}
\hspace{.08\linewidth}
\begin{minipage}{.17\linewidth}
\includegraphics[scale=0.18]{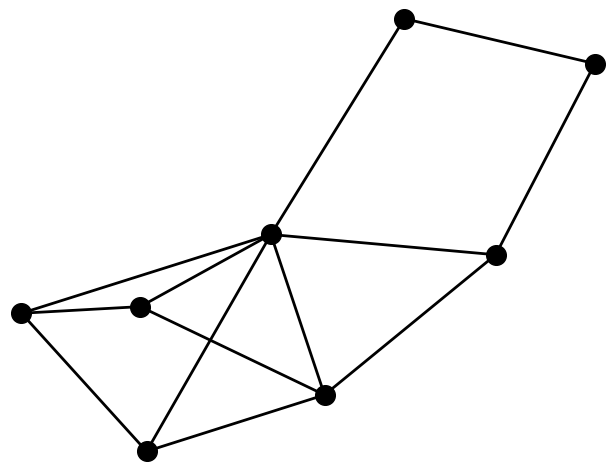}
\end{minipage}
\hspace{.08\linewidth}
\begin{minipage}{.17\linewidth}
\includegraphics[scale=0.18]{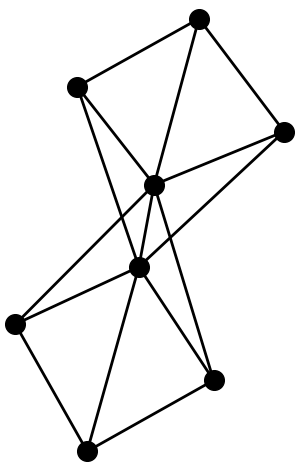}
\end{minipage}
\hspace{.02\linewidth}
\begin{minipage}{.17\linewidth}
  \includegraphics[scale=0.18]{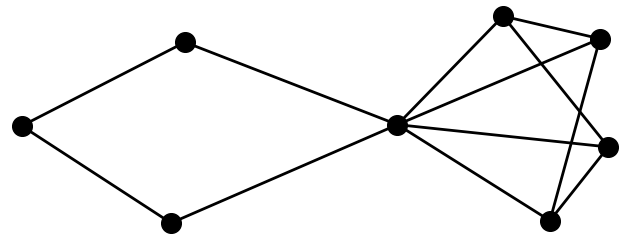}
\end{minipage}
\hspace{.05\linewidth}

\begin{minipage}{.17\linewidth}
  \includegraphics[scale=0.2]{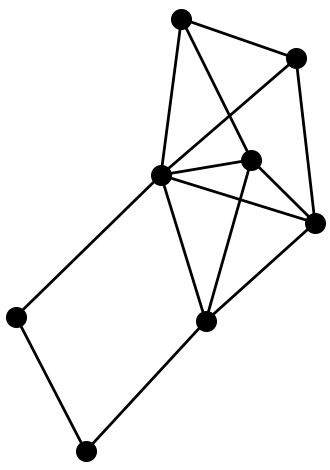}
\end{minipage}
\hspace{.02\linewidth}
\begin{minipage}{.17\linewidth}
\includegraphics[scale=0.18]{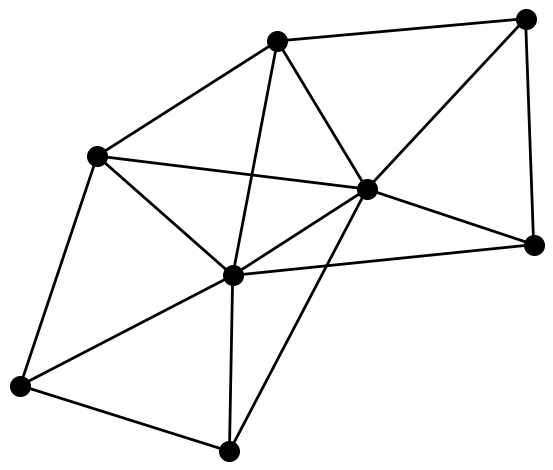}
\end{minipage}
\hspace{.08\linewidth}
\begin{minipage}{.17\linewidth}
\includegraphics[scale=0.2]{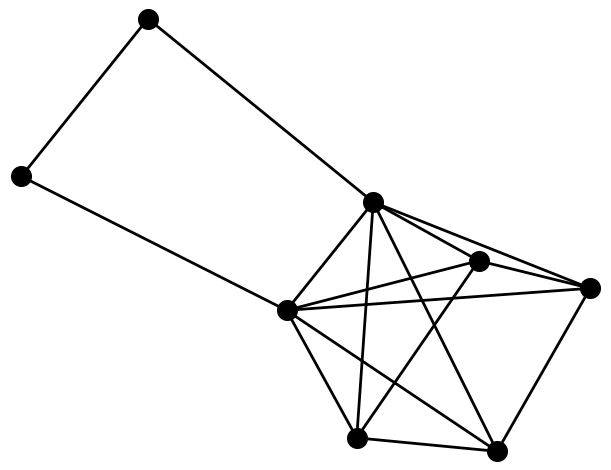}
\end{minipage}
\hspace{.08\linewidth}
\begin{minipage}{.17\linewidth}
  \includegraphics[scale=0.2]{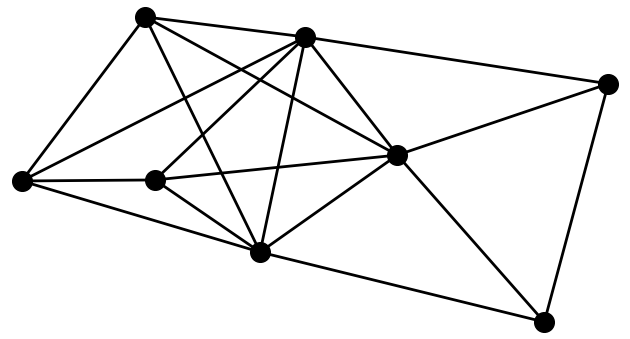}
\end{minipage}
\hspace{.05\linewidth}

\begin{minipage}{.17\linewidth}
  \includegraphics[scale=0.22]{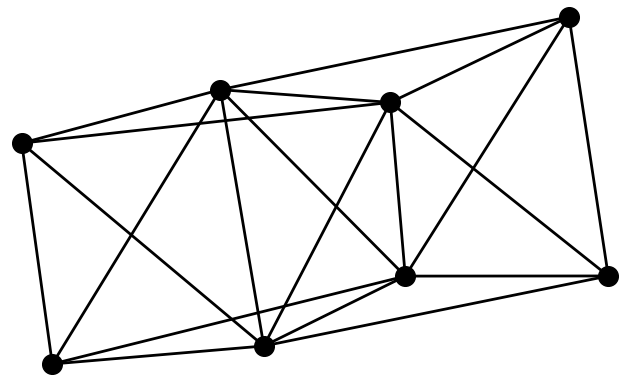}
\end{minipage}
\hspace{.08\linewidth}

\caption{$8$-vertex forbidden induced subgraphs for edge-add chordal graphs.}

\label{chordal_8}

\end{figure}

\section{Forbidden induced subgraphs for edge-add split graphs and edge-add threshold graphs}

A graph is a \textbf{cograph} if it does not contain $P_4$, the 4-vertex path graph.
A graph is \textbf{split} if its vertex set can be partitioned into a clique (a set of pairwise adjacent vertices) and an independent set (a set of pairwise non-adjacent vertices). A graph is \textbf{threshold} if it is a cograph and contains neither the $4$-cycle nor its complement. 

The authors in \cite{vsin} characterized edge-apex cographs by their forbidden induced subgraphs. Since the class of cographs is closed under complementation, an immediate consequence of Corollary \ref{complementation} is a characterization of edge-add cographs by their forbidden induced subgraphs.

We prove the following.

\begin{theorem}
\label{edge_add_threshold}
Let $G$ be a forbidden induced subgraph for the class of edge-add threshold graphs. Then $4 \leq |V(G)| \leq 8$, and $G$ is isomorphic to one of the graphs in Figure \ref{threshold}. 
\end{theorem}

\begin{proof}
Since $G$ is not a threshold graph, it contains a forbidden induced subgraph $H$ for the class of threshold graphs. It is  well-known that the forbidden induced subgraphs for threshold graphs are exactly $2K_2$, $C_4$, and $P_4$. We consider these three exhaustive cases for $H$.  Observe that $2K_2$ is not an edge-add threshold graph since no single edge can be added to it to obtain a threshold graph. Therefore, if $H = 2K_2$, then by minimality $G = 2K_2$, yielding $|V(G)| = 4$. 

Next, suppose that $H = C_4$. Because $H$ has $c=4$ vertices and $k=2$ non-edges, and the maximum number of vertices in a forbidden induced subgraph for threshold graphs is $m=4$, it follows directly by Proposition \ref{unique_FIS} that $|V(G)| \leq 4 + 2(4-2) = 8$.

Finally, suppose that $H$ is an induced path on four vertices $abcd$. The graph $H$ has exactly three non-edges: $ac, bd$, and $ad$. Note that in the graph $G+ad$, the vertex set $V(H) = \{a,b,c,d\}$ induces a $C_4$, which is a forbidden induced subgraph for threshold graphs. Therefore, we may choose $F_{ad} = V(H)$. Because $F_{ad}$ is entirely contained within $V(H)$, it contributes zero vertices outside of $H$. By Lemma \ref{unique_FIS_lemma}, $V(G) = V(H) \cup F_{ac} \cup F_{bd} \cup F_{ad}$. Because $F_{ad}$ contributes zero new vertices, and the sets $F_{ac}$ and $F_{bd}$ each contribute at most $m-2 = 2$ vertices outside of $V(H)$, we obtain $|V(G)| \leq 4 + 0 + 2 + 2 = 8$. 

Because the number of vertices in any such obstruction $G$ is  bounded by 8, applying Algorithm 1 yields the exact and complete list of forbidden induced subgraphs for edge-add threshold graphs, as presented in Figure \ref{threshold}.
\end{proof}

The following is immediate by complementation.

\begin{corollary}
Let $G$ be a forbidden induced subgraph for the class of edge-apex threshold graphs. Then $4 \leq |V(G)| \leq 8$, and $G$ is isomorphic to the complement of one of the graphs in Figure \ref{threshold}. 
\end{corollary}

Next we prove the following. 

\begin{theorem}
Let $G$ be a forbidden induced subgraph for the class of edge-add split graphs. Then $5 \leq |V(G)| \leq 8$, and $G$ is isomorphic to one of the graphs in Figures \ref{split_5}, \ref{split_6}, and \ref{split_7_8}.
\end{theorem}

\begin{proof}
It is clear that $|V(G)| \geq 5$. Since $G$ is not a split graph, it contains a forbidden induced subgraph for split graphs. It is well-known that the forbidden induced subgraphs for split graphs are exactly $2K_2$, $C_4$, and $C_5$.  Suppose $G$ contains a $5$-cycle $C_5$. Observe that $C_5$ is not an edge-add split graph since no edge can be added to it to obtain a split graph. It follows that $G = C_5$ so we may assume that $G$ does not contain a $C_5$. 

\begin{sublemma}
\label{P_5_case}
If $G$ contains a $P_5$, the path on five vertices, then $|V(G)| \leq 8$.
\end{sublemma}

Suppose that $12345$ is a $P_5$ contained in $G$. Observe that the graph $H$ induced by $\{1,2,4,5\}$ is $2K_2$, a forbidden induced subgraph for split graphs. Note that $\{1,2,3,4,5\}$ induces a $C_5$ in $G+15$ where $15$ is a non-edge of $H$. Similarly, $\{2,3,4,5\}$ and $\{1,2,3,4\}$ induce $C_4$ in $G+25$ and $G+14$ respectively. Let $F_{24}$ denote the subset of $V(G)$ such that $(G+24)[F_{24}]$ is a forbidden induced subgraph for split graphs. By Lemma \ref{unique_FIS_lemma}, it follows that $V(G) = \{1,2,3,4,5\} \cup F_{24}$. Observe that if $F_{24} - \{1,2,3,4,5\}$ exceeds three, then $G[F_{24}]$ is a forbidden induced subgraph for split graphs and the result follows by Lemma \ref{unique_FIS_lemma_disjoint_case}. Thus Sublemma \ref{P_5_case} holds. 

Next suppose that $G$ contains a $4$-cycle $C_4$, say $H = abcda$. By Lemma \ref{unique_FIS_lemma}, corresponding to each non-edge $ac$ and $bd$ of $H$, we have a subset $F_{ac}$ and $F_{bd}$ such that $(G+ac)[F_{ac}]$ and $(G+bd)[F_{bd}]$ are forbidden induced subgraphs for split graphs. Moreover, $V(G) = \{a,b,c,d\} \cup F_{ac} \cup F_{bd}$. Observe that if both $F_{ac} - \{a,b,c,d\}$ and $F_{bd} - \{a,b,c,d\}$ are at most two, then $|V(G)| \leq 8$ so we may assume that $|F_{ac} - \{a,b,c,d\}| \geq 3$. 

Note that if $|F_{ac}| = 4$, because $|F_{ac} - \{a,b,c,d\}| \geq 3$, the set $F_{ac}$ can share at most one vertex with $\{a,b,c,d\}$. 
Thus, $F_{ac}$ cannot contain both $a$ and $c$, that is the edge $ac$ is not present in $(G+ac)[F_{ac}]$. Consequently, $(G+ac)[F_{ac}] = G[F_{ac}]$. Since $C_4$ and $F_{ac}$ have at most one common vertex, it follows by Lemma \ref{unique_FIS_lemma_disjoint_case} that $|V(G)| \leq 8$. Therefore $|F_{ac}| = 5$. Since $G$ contains no $C_5$, $G[F_{ac}]$ is an induced path on five vertices and the result follows by Sublemma \ref{P_5_case}. 

We may now assume that $G$ contains no $C_4$ or $C_5$ so $G$ contains a forbidden induced subgraph $H = 2K_2$ for the class of split graphs. Suppose this $H$ has edges $12$ and $34$ so the non-edges are $\{13,23,14,24\}$. Note that corresponding to each non-edge $h$ of $H$, there is a subset $F_h$ of $V(G)$ such that $(G+h)[F_h]$ is a forbidden induced subgraph for the class of split graphs. Observe that, if for any $h$ in $\{13,23,14,24\}$, the graph $(G+h)[F_h]$ is a $C_5$, then $G[F_h]$ must be a $P_5$ since $G$ contains no $C_5$. The result then follows by Sublemma \ref{P_5_case}. 

Therefore for each $h$ in $\{13,23,14,24\}$, the graph $(G+h)[F_h]$ is isomorphic to $C_4$, or $2K_2$. If for every such $F_h$, we have $|F_h - \{1,2,3,4\}| \leq 1$, then by Lemma \ref{unique_FIS_lemma}, the result follows. We choose $F_h$ such that $|F_h - \{1,2,3,4\}|$ is maximal, say $F_h = F_{13}$. We may assume that $F_{13} - \{1,2,3,4\}$ is at least two. 

First suppose that $(G+13)[F_{13}]$ is a $C_4$. If $F_{13} \cap \{1,2,3,4\} \neq \{1,3\}$, then the result follows by Lemma \ref{unique_FIS_lemma_disjoint_case} so $F_{13} = \{1,3,a,b\}$ and $1ab3$ is an induced path of length four in $G$. Observe that if $a4$ is not an edge of $G$, then $\{a,1,3,4\}$ induce a forbidden induced subgraph $F=2K_2$ in $G$. Since $F$ is an induced subgraph of $G+h$ for $h$ in $\{14,23,24\}$, it follows by Lemma \ref{unique_FIS_lemma} that $|V(G)| = 6$. Similarly, $b2$ is an edge of $G$. Moreover, if $a2$ is not an edge in $G$, then $\{a,1,2,b\}$ induce a $C_4$ in $G$, a contradiction to our assumption. Therefore $a2$ is an edge of $G$. Similarly, $b4$ is an edge of $G$. Observe that $\{a,2,3,4\}$ and $\{1,2,b,4\}$ induce  $C_4$ in $G+23$ and $G+14$ respectively. By Lemma \ref{unique_FIS_lemma}, it follows that $V(G) = \{1,2,3,4,a,b\} \cup F_{24}$. Since $|F_{24}| = 4$ and $|F_{24} - \{1,2,3,4\}| \leq 2$, the result follows.

Finally suppose that $(G+13)[F_{13}]$ is a $2K_2$. By a similar argument as above, $F_{13} \cap \{1,2,3,4\} = \{1,3\}$, say $F_{13} = \{a,b,1,3\}$. Observe that $G[F_{13}]$ has only one edge, namely $ab$. Note that if $2$ has no neighbors in $\{a,b\}$ in $G$, then $\{1,2,a,b\}$ induce a forbidden induced subgraph $F = 2K_2$ in $G$. Since $F$ is a forbidden induced subgraph in $G+h$ for each $h$ in $\{13, 14, 23, 24\}$, by Lemma \ref{unique_FIS_lemma}, we obtain $|V(G)| \leq 6$. Therefore $2$ has a neighbor in $\{a,b\}$ in $G$. Similarly, $4$ has a neighbor in $\{a,b\}$ in $G$. If $2$ and $4$ have a common neighbor in $\{a,b\}$, say $a$, then $12a43$ is an induced $P_5$ in $G$ so the result follows by \ref{P_5_case}. Therefore $a4$ and $b2$ are non-edges of $G$, while $a2$ and $b4$ are edges of $G$. We again obtain an induced $P_5$ in $G$ and the result follows by \ref{P_5_case}.  

Since we have bounded the order of any such graph to at most 8 vertices, we apply Algorithm 1 using SageMath \cite{sage} to obtain the complete list of forbidden induced subgraphs for edge-add split graphs, presented in Figures \ref{split_5}, \ref{split_6}, and \ref{split_7_8}.
\end{proof}

 The following is immediate by taking complements.

\begin{corollary}
Let $G$ be a forbidden induced subgraph for the class of edge-apex split graphs. Then $5 \leq |V(G)| \leq 8$, and $G$ is isomorphic to the complement of one of the graphs in Figures \ref{split_5}, \ref{split_6}, and \ref{split_7_8}. 
\end{corollary}

A graph $G$ is a {\bf $(p,q)$-split graph} if $V(G)$ can be partitioned into $V_1$ and $V_2$ such that $G[V_1]$ has independence number (size of its largest independent set) at most $p$, and $G[V_2]$ has  clique number at most $q$. Gy\'{a}rf\'{a}s \cite{gyarfas} showed that for fixed $p,q$, the class of $(p,q)$-split graphs is characterized by a finite set of forbidden induced subgraphs. Chudnovsky and Seymour \cite{CS} proved that every forbidden induced subgraph for the class of $(k,k)$-split graphs has at most $(k+1)^{2^{2k+1}}$ vertices. 

Following an analogous framework, we define a graph $G$ to be a {\bf $(p,q)$-edge split graph} if $V(G)$ can be partitioned into $V_1$ and $V_2$ such that $G[V_1]$ can be obtained from a clique by deleting at most $p$ edges, and $G[V_2]$ has at most $q$ edges. In other words, $G[V_1]$ is at most $p$ edge-additions away from being a clique while $G[V_2]$ is at most $q$ edge-deletions away from being an independent set. 

Observe that edge-add split graphs are $(1,0)$-edge split graphs while edge-apex split graphs are $(0,1)$-edge split graphs. Since the operations of edge-deletion and edge-addition commute, we obtain the following characterization of $(p,q)$-edge split graphs.

\begin{lemma}
\label{split_add_apex}
Let $\mathcal{S}$ be the class of split graphs, and let $p,q$ be nonnegative integers. Then the class of $(p,q)$-edge split graphs is equal to the $p$-edge-add class of the $q$-edge-apex class of $\mathcal{S}$. 
\end{lemma}

The following is now immediate using Corollaries \ref{q_edge_apexing} and \ref{p_edge_add}, and  Lemma \ref{split_add_apex}.

\begin{theorem}
The set of forbidden induced subgraphs for the class of $(p,q)$-edge split graphs is finite.
\end{theorem}

\end{document}